\documentclass[11pt,reqno]{amsart}
\usepackage[T2A]{fontenc}
\usepackage[cp1251]{inputenc}
\usepackage[russian,english]{babel}

\usepackage{izvuz6}
\usepackage[pdftex,unicode,colorlinks,linkcolor=blue,
citecolor=red,bookmarksopen,pdfhighlight=/N]{hyperref}
\usepackage{verbatim}
\usepackage{amssymb,amsmath,amsthm,amsfonts}
\usepackage{graphicx}
\usepackage{enumerate}
\theoremstyle{plain}
\newtheorem{theorem}{Теорема}
\newtheorem*{theoJ}{Неравенство Йенсена}
\newtheorem*{theoremH}{Теорема Хёрмандера}

\newtheorem{propos}{Предложение}
\newtheorem{corollary}{Следствие}

\theoremstyle{definition}

\newtheorem{remark}{Замечание}
\newtheorem{example}{Пример}

\newcommand{\NN}{\mathbb{N}}
\newcommand{\RR}{\mathbb{R}}
\newcommand{\CC}{\mathbb{C}}
\newcommand{\const}{{\rm const}}

\renewcommand{\Im}{{\rm Im \,}}
\renewcommand{\leq}{\leqslant}
\renewcommand{\geq}{\geqslant}

\DeclareMathOperator{\im}{im}
\DeclareMathOperator{\dist}{dist}
\DeclareMathOperator{\Int}{int}
\DeclareMathOperator{\Hol}{Hol}
\DeclareMathOperator{\clos}{clos}
\DeclareMathOperator{\loc}{loc}

 \numberwithin{equation}{section}  

\tit{\MakeUppercase{От интегральных оценок  функций к равномерным и локально усредненным}}
\shorttit{\MakeUppercase{От интегральных оценок  функций к равномерным и локально усредненным}} 
\author{\MakeUppercase{Р.\,А.~Баладай, Б.\,Н.\,Хабибуллин}}
\tauthor{\MakeUppercase{Р.\,А.~Баладай, Б.\,Н.\,Хабибуллин}}  

\god{201\_}         

\pp{\pageref{firstpage}--\pageref{lastpage}}

\begin{document}
\selectlanguage{russian}
\maketit
\label{firstpage}

\abstract{В теории функций комплексных переменных нередко возникают задачи поточечной оценки сверху функции или ее усреднений при известных интегральных ограничениях на рост этой функции. Мы предлагаем вариант подхода к таким задачам, основанный на интегральном неравенстве Йенсена с выпуклой функцией.} 


\udk{517.5}  

\englishabstract{Problems  pointwise estimates from above functions or its averages often arise in the function theory under known integral restrictions on the growth of this function. We offer an approach to such problems based on the integral Jensen's inequality with the convex function.}

\englishkeywords{holomorphic function, average value, convex function, (pluri-)subharmonic function, $\bar\partial$-problem, integral Jensen's inequality, measure space}  

\Received{\_.\_.201\_}  

\thnk{Работа выполнена при финансовой поддержке Российского фонда
фундаментальных исследований, проект 16-01-00024а.}


\section{Введение}

\subsection{О  результатах}\label{s00} 
Необходимость оценок поточечного роста голоморфной или иной фу\-н\-к\-ции $f$ комплексных переменных при a priori известной или заданной глобальной интегральной оценке на нее (с весом $v$) --- довольно часто возникающая задача. Один из наиболее распространенных способов перехода от интегральной оценке с весом к поточечной, или равномерной, в самых общих ситуациях состоит в использовании (плюри-)субгармоничности функции $|f|$ или $\ln |f|$ и оценке сверху через локальную  точную 
 верхнюю грань по шарам, полидискам (кругам),  сферам (окружностям) веса $v$ с <<малыми>> добавками. Этот способ как промежуточный этап применялся, например, в статьях  А.\,В.~Абанина \cite[лемма 4]{Ab1986}, \cite[\S~1, лемма]{Ab1987}
и его диссертации \cite[леммы 1.3.1, 1.3.3]{Abd}, в совместной работе  К.\,Д. Бирштедта, Х. Бонета, Я. Таскинена  \cite[{\bf 3.C}, предложение 3.8]{BBT} и др., часто с дополнительными требованиями, которые здесь не обсуждаются. 
Кратко такой переход проиллюстрирован в \cite[следствие HB]{BKhFa}, \cite[лемма 9.1]{Kh01}  и в замечании \ref{r:rr} во вспомогательных целях  для  примера \ref{ex:c}.
Для одной комплексной переменной более тонкие оценки сверху через локальные усреднения веса $v$ по окружностям установлены и использованы  в работах О.\,В.~Епифанова \cite{Ep90}, \cite{Ep92}, диссертации 
А.\,В.~Абанина  \cite[лемма 1.3.2]{Abd} и в совместной статье Т.\,Ю~Байгускарова, 
Г.\,Р.~Талиповой и Б.\,Н.~Хабибуллина  \cite[теорема 2]{BTKhA};  через усреднения по кругам для одной переменной 
и по шарам для  многих --- в совместных  статьях Т.\,Ю~Байгускарова и Б.\,Н.~Хабибуллина соответственно   \cite{KhB16}  и \cite[теоремы 1,2]{BKhFa}. Дополняют их результаты из п.~\ref{3_1},
где приведен пример \ref{ex:Fs}, демонстрирующий  степень точности теоремы \ref{est:hf}, и сформулирована 
 теорема \ref{th:pfvv}, обобщающая теорему  \ref{est:hf}.  Теорема \ref{th:pfvv} --- следствие теоремы \ref{th:pfv-} из п.~\ref{4_1}.

Близкий вопрос --- вывод равномерных оценок решений $\bar{\partial}$-задачи  из интегральных. C рядом результатов 
о равномерных и локально усредненных оценках сверху решений этой задачи через правую часть  в случае одной переменной можно ознакомится  по статьям О.\,В.~Епифанова \cite[теорема 2, следствие]{Ep90}, \cite[теоремы 1,2]{Ep92} и  Х.~Ортега-Серды \cite{JC} и билиографии в них. В настоящей статье мы также  даем локально усредненную  оценку роста решений $\bar{\partial}$-задачи в теореме  \ref{cor:iH}  из пп.~\ref{3_2}, доказанной в разделе \ref{s4}.

В  упомянутых статьях \cite{KhB16} и \cite{BKhFa} было использовано интегральное неравенство Йенсена. Здесь мы 
применяем развитие этого приема. С целью его унификации значительная часть (раздел \ref{s:2}) настоящей статьи посвящена различным формам трактовки неравенства Йенсена (теорема \ref{JId} и другие утверждения и примеры раздела \ref{s:2}), нацеленным именно на применения к оценкам голоморфных функций и решений $\bar{\partial}$-задачи. Из последних известных нам продвижений по совершенствованию интегрального неравенства Йенсена отметим здесь статью  К.\,П.~Никулеску \cite[теорема A]{N2001},  монографию 
К.\,П.~Никулеску  и Л.-Э.~Перссона \cite[теоремы 1.8.1, 1.8.4]{NP06}, совместную  статью З. Павича,  Й.~Печарича и  И.~Перича \cite{PPP11}, статью Л.~Хорвата \cite{LH}, совместную статью Ш.~Абрамовича и Л.-Э.~Перссона \cite{AP}, 
а также библиографию в них. 

Авторы глубоко признательны рецензенту за ряд полезных усовершенствований, замечаний и полезных советов.

\subsection{Основные обозначения, определения и соглашения}
Как обычно,  $\NN$ --- множество натуральных чисел, а $\RR$ и  $\CC$
 --- множества соответственно всех {\it  вещественных\/} и  {\it комплексных чисел\/}  
с их естественными структурами 
  и с соответствующим вложением $\RR\subset \CC$. При $n\in \NN$ пространство $\CC^n$ при необходимости отождествляем с $\RR^{2n}$.

Одним и тем же символом $0$ обозначаем, по контексту, число нуль, нулевой вектор, нулевую функцию, нулевую меру и т.\,п. 
 Для подмножества  $X$ упорядоченного векторного пространства  чисел, функций, мер или т.\,п. с элементом $0$ и отношением порядка $\leq$ полагаем  $X^+:=\{x\in X\colon 0\leq x\}$ --- все положительные элементы из $X$; $x^+:=\max \{0,x\}$. 
{\it Положительность\/} всюду понимается, в соответствии с контекстом, как $\geq 0$, а $>0$ --- {\it строгая положительность.\/} Аналогично для отрицательности.  Функция $f\colon X\to Y$ с упорядоченными  $(X, \leq)$, $(Y, \leq)$ {\it возрастающая}, если для любых $x_1,x_2\in X$ из $x_1\leq x_2$ следует $f(x_1)\leq f(x_2)$, и {\it строго возрастающая,\/} если 
для любых $x_1,x_2\in X$ из $x_1< x_2$ следует $f(x_1)< f(x_2)$. Аналогично для убывания.

Для меры  (функции) $a$ её сужение на множество $X$ обозначаем как $a\bigm|_X$. Через  $\circ$ обозначаем  операцию суперпозиции функций.

Для подмножества $S$ в $\RR$ или в $\CC^n$ через $\clos S:=:\overline S$, $\Int S$, $\partial S$ обозначаем соответственно {\it замыкание\/}, {\it внутренность\/}  и {\it границу\/}  множества $S$. Через $\dist (\cdot , \cdot) $ обозначаем функции евклидова расстояния между парами точек, между  точкой и множеством, а также  
между множествами в $\CC^n$ или в $\RR^n$, $n\in \NN$; $\dist (\cdot , \varnothing):=\inf \varnothing:=+\infty$.   Для  
числа $r>0$ и $x\in \RR^n$ 
\begin{equation}\label{Bxr}
	B(x,r):=\{x' \in \RR^n \colon \dist(x',x)<r\} 
\end{equation}
 ---  открытый шар (интервал при $n=1$ и круг  при $n=2$) с центром $x$ радиуса $r$; $ \overline{B}(x_0,r_0)$ --- его замыкание, а $\partial B(x,r)$ --- сфера с центром $x$ радиуса $t$.

Для открытого множества  $\mathcal O \subset \CC^n$ 	 через  $\Hol (\mathcal O)$ обозначаем  векторное   пространство   над полем $\CC$  	 функций, голоморфных на $\mathcal O$.

Через $\const (a_1, a_2, \dots)\in \RR$ обозначаем постоянные, вообще говоря, зависящие только  от $a_1, a_2, \dots$ и, если не оговорено противное, только от них, но зависимость от размерности $n$ и открытого множества $\mathcal O$ не указываем.

\section{Оценки голоморфных функций и решения $\bar{\partial}$-задачи}\label{s10}

\subsection{Оценки голоморфных функций}\label{3_1}  Значительную роль и  как самостоятельный объект,  и как вспомогательный аппарат в комплексном анализе 
играют классы голоморфных функций $f\in \Hol (\mathcal O)$ на  {\it открытом множестве\/}  $\mathcal O\subset \CC^n$,  удовлетворяющих при некотором $0<p\in \RR$ ограничению вида 
\begin{equation}\label{df:cle}
\|f\|_w:=
\Bigl( \int_{\mathcal O} |f|^p e^{-w}\, d\lambda \Bigr)^{1/p}<+\infty, \quad\text{$\lambda$ --- {\it мера Лебега на\/}  $\mathcal O$},
\end{equation}
 где функция  $w\colon [-\infty,+\infty]\to [-\infty,+\infty]$ такова, что интеграл Лебега в \eqref{df:cle} корректно определен. Функцию $w$ часто называют {\it весовой функцией,\/} или просто  {\it весом.} 

Для   интегрируемой по $\lambda$ в шаре $B(z,r)\subset \CC^n$ функции $w$
введем характеристику
\begin{equation}\label{df:B}
	B_w(z,r)\overset{\eqref{Bxr}}{:=} \frac{1}{\lambda\bigl(B(z,r)\bigr)} \int\limits_{B(z,r)} w\,d \lambda=\frac{n!}{\pi^nr^{2n}}\int\limits_{B(z,r)} w\,d \lambda 
\end{equation}
--- {\it среднее функции\/ $w$ по шару\/ $B(z,r)$.} 
Имеет место 
\begin{theorem}\label{est:hf} Пусть выполнено \eqref{df:cle}. Тогда для всех $z\in \mathcal O$
\begin{equation}\label{esusr}
	{\,\ln\,}\bigl|f(z)\bigr|\overset{\eqref{df:B}}{\leq} \frac{1}{p}\,\inf_{0<r<\dist (z,\CC^n\setminus \mathcal O)}\left(B_w(z,r)			+	{2n}\, {\,\ln\,} \frac{1}{r}\right) +{\,\ln\,} \|f\|_w+\frac{1}{p}{\,\ln\,} \frac{n!}{\pi^n}. 
	\end{equation}
\end{theorem}
Информацию об уровне  точности теоремы \ref{est:hf} в одном частном случае дает 
\begin{example}\label{ex:Fs} В случае $\mathcal O=\CC$, $w(z)=|z|^2$, $z\in \CC$, $p=2$ для класса целых функций $f\in \Hol (\CC)$, удовлетворяющих условию \eqref{df:cle}, известного как  пространство Фока \cite{Zhu}, имеет место точная 
оценка \cite[теорема 2.7, следствие 2.8]{Zhu}
\begin{equation}\label{est:F}
	\bigl|f(z)\bigr|\leq\frac{1}{\sqrt \pi} \|f\|_w e^{|z|^2/2}, \quad\text{где } 
	w(z)=|z|^2 , \; z\in \CC.
\end{equation}
 В то же время из  оценки \eqref{esusr} теоремы \ref{est:hf}  для таких функций $f$ имеем
\begin{multline}\label{Foxsr}
	{\,\ln\,}\bigl|f(z)\bigr|{\leq} \frac{1}{2}\,\inf_{r>0}\left(\frac{1}{\pi r^2}\int_0^r \int_0^{2\pi}|z+te^{i\theta}|^2t\,d \theta	\,dt			+	{2}\, {\,\ln\,} \frac{1}{r}\right) +{\,\ln\,} \|f\|_w+\frac12{\,\ln\,} \frac{1}{\pi}\\
=\Bigl|\;\text{$\inf$ достигается при $r=\sqrt{2}$}\;\Bigr|={\,\ln\,} \frac{1}{\sqrt{\pi}}+
{\,\ln\,} \|f\|_w +	\frac{1}{2}\,|z|^2 +\ln \sqrt{\frac{e}{2}} \, . 
\end{multline}
Правая часть  в \eqref{Foxsr} после потенцирования больше правой части неравенства из \eqref{est:F}  на числовой множитель $\sqrt{ e/2}\leq  1{,}17$, что иллюстрирует неулучшаемость оценки \eqref{esusr} по порядку роста.
Результатов с точностью уровня  классической оценки \eqref{est:F} для  весьма специального случая пространства Фока целых функций применительно к  теореме \ref{est:hf} для произвольных весов $w$ и  множеств $\mathcal O\subset \CC^n$, вообще говоря, ожидать не приходится. 
\end{example}

Класс функций, удовлетворяющих ограничению \eqref{df:cle}, может быть обобщен до  класса функций $f\in \Hol (\mathcal O)$, удовлетворяющих условию 
\begin{equation}\label{fv}
	N_{\Phi}(f;v):=\int_{\mathcal O} \Phi \bigl({\,\ln\,} |f|- v\bigr)\, d\lambda<+\infty, 
\end{equation}
где $\Phi, v \colon [-\infty,+\infty]\to [-\infty,+\infty]$  --- функции, для которых интеграл Лебега в \eqref{fv} корректно определен.
Здесь функцию  $v$ также называем {\it весом.\/} В частном случае 
\begin{equation}\label{Phivw}
	\Phi(x):=e^{px}, \ x\in [-\infty, +\infty], \; e^{-\infty}:=0,\; e^{+\infty}:=+\infty, \quad v:=\frac1p w 
\end{equation}
интеграл в \eqref{fv} --- это в точности  интеграл из \eqref{df:cle}.
\begin{theorem}\label{th:pfvv}
Для  выпуклой строго\footnote{Условие строгого возрастания функции $\Phi$  будет  снято   в теореме \ref{th:pfv-}
из п.~\ref{4_1}.}  возрастающей функции $\Phi$ из \eqref{fv}     следует 
\begin{equation}\label{esusr+}
	{\,\ln\,}\bigl|f(z)\bigr|\overset{\eqref{df:B},\eqref{fv}}{\leq} \inf_{0<r<\dist (z,\CC^n\setminus \mathcal O)}
	\left(B_v(z,r)			+	\Phi^{-1} \Bigl( \frac{\pi^n}{n!}r^{2n} N_{\Phi}(f;v)\Bigr)\right).
	\end{equation}
\end{theorem}
При выборе функций $\Phi$ и $v$ как в \eqref{Phivw} из теоремы \ref{th:pfvv} сразу получаем теорему  \ref{est:hf}.

\begin{remark} В случае  веса $v$ в $\CC^n$, субгармонического в $\mathcal O$, средние по шарам $B_v(z,r)$ в \eqref{esusr+} и, с учетом обозначений \eqref{Phivw}, в \eqref{esusr}   можно заменить на не  м\'еньшие средние по сферам 
\begin{equation*}
	S_{v}(z,r):=\frac{(n-1)!}{2\pi^n\max\{1,2(n-1)\}r^{2n-1}}\int_{\partial B(0,1)} v(z+rz')\,d \sigma_{2n-1}(z')\geq B_v(z,r),
\end{equation*}
где $d\sigma_{2n-1}$ --- элемент площади на единичной сфере $\partial B(0,1)$.
\end{remark}
\subsection{Локально усредненные  оценки решения $\bar\partial$-задачи}\label{3_2} Через $L^2_{\loc}(\mathcal O)$ обозначаем класс локально интегрируемых 
по мере Лебега $\lambda$ функций на открытом множестве из $\mathcal O\subset \RR^n$ или $\CC^n$; $L^2_{(0,1)}$  ---пространство $(0,1)$-форм с коэффициентами из $L^2_{\loc}(\mathcal O)$.

\begin{theorem}\label{cor:iH} 
Пусть $\mathcal O$ --- псевдовыпуклое открытое множество в $\CC^n$, $v$ --- плюрисубгармоническая функция в $\mathcal O$ и $0<a\in \RR$, а также $g\in L^2_{(0,1)}$ и $\bar\partial g=0$. Тогда уравнение 
$\bar\partial f=g$  имеет решение $f\in L^2_{\loc}(\mathcal O)$, которое при всех 
\begin{subequations}\label{efvaz}
\begin{align}
	z\in \mathcal O \quad &\text{ и } \quad 0<r<\min \bigl\{ 1, \dist (z, \CC^n\setminus \mathcal O) \bigr\}
\tag{\ref{efvaz}d}\label{co:rd}\\
\intertext{удовлетворяет  оценке}
	B_{{\ln\,}|f|}\bigl(z,r\bigr) &\leq 
	\frac12\, B_v(z,r) 	+a{\,\ln\,} \bigl(1+|z|\bigr) +	n{\,\ln\,} \frac{1}{r}
	+\frac{1}{2}\,\ln J(g,v)+\const(a), 
	\tag{\ref{efvaz}b}\label{co:ra}
	\\
	\intertext{где }
	J(g,v) &:=\int\limits_{\mathcal O}\bigl|g(z) \bigr|^2e^{-v(z)}\bigl(1+|z|^2\bigr)^{2-a}
\,d \lambda (z)<+\infty. 
\tag{\ref{efvaz}g}\label{co:r}
	\end{align}
\end{subequations} 

\end{theorem}

\section{К интегральной форме неравенства Йенсена для выпуклых функций}\label{s:2}  
В этом  разделе\/ $(X, \mu)$ --- пространство с положительной  счётно-аддитивной мерой\/ $\mu\neq 0$ \cite{Halm}, \cite{Bog}; $L^1(X,\mu)$ --- класс функций $f$, определённых на $X$ почти всюду (п.\,в.) по мере $\mu$  значениями из $\RR$ и 
интегрируемых по $\mu$ на $X$;  $I\subset \RR$ --- связное непустое множество, т.\,е.  интервал ограниченный или нет, открытый или  же содержащий  одну или обе свои граничные точки из $\RR$ (вырожденный случай $I$ --- одноточечное множество). 
{\it Выпуклую функцию\/} $\Phi\colon I\to \RR$ на $I$ рассматриваем в <<обычном>> смысле\footnote{Другие формы определения выпуклой функции не всегда эквивалентны приведенной. Например, если определять выпуклую функцию как верхнюю огибающую аффинных функций, то для замкнутого интервала-отрезка $I$ такая функция всегда непрерывна в отличие от <<обычной>> выпуклой функции.}:
\begin{equation*}
\Phi\bigl(at_1+(1-a)t_2\bigr)\leq a\Phi(t_1)+(1-a)\Phi(t_2)\quad \text{для всех $t_1,t_2\in I$ и $0\leq a\leq 1$.}
\end{equation*}
в частности, $\Phi\bigm|_{\Int I}\in C(\Int I)$, а при $+\infty >\sup I\in I$ (соответственно $-\infty <\inf I\in I$)
\begin{equation*}
\Phi (\sup I)\geq \lim_{\Int I \ni t\to \sup I}\Phi (t) \quad \bigl(\text{соответственно $\Phi (\inf I)\geq \lim_{\Int I \ni t\to \inf  I}\Phi (t)$}\bigr),
\end{equation*}
где всегда существуют конечные пределы. В этих обозначениях имеет место 
 \begin{theoJ}[см. {\cite[теорема A]{N2001}, \cite[теоремы 1.8.1, 1.8.4]{NP06}}]\label{JI}  Пусть $\mu (X)=1$, т.\,е. $(X, \mu)$ --- вероятностное пространство,   $f\in L^1(X,\mu)$ и $f(x)\in I$ для почти   всех  точек $x\in X$ по мере $\mu$.  Тогда 
\begin{equation}\label{inJ}
	\int_X f \, d\mu \in I \quad \text{ и }\quad  \Phi \Bigl(\int_{X} f\,d \mu\Bigr)\leq \int_{X} \Phi\circ f \, d\mu 
\end{equation}
при условии, что $\Phi\circ f \in L^1(X,\mu)$.
 \end{theoJ}

\subsection{$\sup$-обратная функция} 
Для произвольной функции $\Phi$  через $\im \Phi$ обозначаем ее образ.
Для произвольной  функции $\Phi\colon I \to \RR$ 
корректно определена {\it $\sup$-обратная функция\/} 
\begin{equation}\label{supf}
\sup \Phi^{-1}\colon y\mapsto \sup \Bigl(\Phi^{-1}\bigl(\{y\}\bigr)\Bigr), \quad   y\in \im \Phi .
\end{equation}
Построение графика функции $\sup \Phi^{-1}$ соcтоит из двух шагов: 1)  симметричное отображение графика функции $\Phi$ относительно биссектрисы первой и третьей четверти; 2) построение наименьшей верхней огибающей полученного на предыдущем шаге  образа кривой с проекцией $\im \Phi\subset  \RR$ на оси абсцисс.  
Из определения \eqref{supf}  сразу следует 
\begin{propos}\label{pr:b}
В условиях теоремы\/ {\rm \ref{JI}} в случае возрастающей  функции $\sup \Phi^{-1}$ и связного множества, т.\,е. интервала, $\im \Phi$, справедливо неравенство
\begin{equation}\label{inJv}
\int_{X} f\,d \mu\leq \sup \Phi^{-1}\Bigl(\,\int_{X} \Phi\circ f \, d\mu\Bigr), 
\end{equation}
где в случае строгого возрастания  функции\/ $\Phi$, т.\,е. её обратимости, точную верхнюю грань\/ $\sup$ в правой части можно убрать.
\end{propos}
\begin{remark}\label{r:r} Возрастание $\sup \Phi^{-1}$ нужно для сохранения неравенства $\leq$ при переходе от  \eqref{inJ}
к \eqref{inJv}. Связность $\im \Phi$ необходима для того, чтобы значение интеграла в правой части \eqref{inJv} попало в область определения функции $\sup \Phi^{-1}$, и гарантирована для открытого интервала $I$ или при $\Phi\in C(I)$, но не только. Например, для выпуклой  на замкнутом интервале $[-1,3]$ разрывной функции
\begin{equation*}
	\Phi (t)=
	\begin{cases}
	|t| &\quad\text{при $-1<t\leq 3$},\\
		2 &\quad\text{при $t=-1$},
	\end{cases}
\end{equation*}
имеем $\im \Phi=[0,3]$, $\sup \Phi^{-1}(y)=y$ для всех $y\in [0,3]$ и, очевидно,  $\sup \Phi^{-1}$ строго возрастает.
\end{remark}
В свете предложения \ref{pr:b} актуально
\begin{propos}\label{co:sup} Пусть $\Phi\colon I \to \RR$ --- выпуклая функция. Следующие два утверждения эквивалентны:
\begin{enumerate}[{\rm I.}]
	\item  Образ $\im \Phi\subset \RR$ --- связное множество, т.е. интервал, а функция $\sup \Phi^{-1}$ возрастающая на $\im \Phi$.
	\item Выполнено одно из  трех условий
		\begin{enumerate}[{\rm 1)}]
		\item $\Phi$ --- постоянная функция на $I$  (в частности, это  так, когда $I$ --- одноточечное множество) и тогда $\im \Phi=\Phi(I) $ --- одноточечное множество, а $\sup \Phi^{-1}\bigl(\Phi(I)\bigr) =\{\sup I\}$ --- одноточечное множество;
		\item  $\Phi$ --- строго возрастающая функция на $I$, непрерывная в точке $\sup I$, если $+\infty\neq \sup I\in I$,  и тогда $\im \Phi$ --- интервал с м\'еньшей граничной точкой $\lim\limits_{\Int I\ni t\to \inf I} \Phi(t)$, содержащей её, если $-\infty \neq \inf I\in I$, и с б\'ольшей  граничной точкой $\lim\limits_{\Int I\ni t\to \sup I} \Phi(t)$, содержащей её, если $+\infty \neq \sup I\in I$, а $\sup \Phi^{-1}=\Phi^{-1}\colon \im \Phi \to I$ --- стандартная  обратная к $\Phi$ функция; 
\item функция $\Phi$ --- непостоянная, ограничена снизу и имеется точка
\begin{equation*}
t_{\max}:=\sup \bigl\{t\in \Int I\colon \Phi (t)=\min_{s\in I}  \Phi(s)\bigr\}\in \Int I,
\end{equation*}
для которых имеют место свойства
\begin{enumerate}[{\rm i)}]
	\item сужение $\Phi\bigm|_{I\setminus (-\infty,t_{\max})}$ --- непрерывная и строго возрастающая функция;
	\item если $I\subset \RR$ --- открытый интервал или $+\infty\neq \sup I\in I$, то 
	\begin{equation*}
	\limsup_{I\ni t\to \inf I} \Phi(t)\leq 	\lim_{I\ni t\to \sup I} \Phi(t);
			\end{equation*}
\item если $-\infty\neq \inf I\in I$, но $\sup I\notin I$, то 
\begin{equation*}
	\limsup_{I\ni t\to \inf I} \Phi(t)<	\lim_{I\ni t\to \sup I} \Phi(t).
			\end{equation*}
\end{enumerate}
	\end{enumerate}
При этом $\im \Phi=I\setminus (-\infty, t_{\max})$, а $\sup \Phi^{-1}$ --- функция, обратная к строго возрастающему сужению   $\Phi\bigm|_{I\setminus (-\infty,t_{\max})}$.
\end{enumerate}
Более того, при выполнении  условия\/ $\rm I$, или эквивалентного ему\/  $\rm II$, функция $\sup \Phi^{-1}$  строго возрастающая и вогнутая на $\im I$.
\end{propos}
Доказательство опускаем. Наиболее просто и наглядно оно следует из графическо-ге\-о\-м\-е\-т\-р\-и\-ч\-е\-с\-ких соображений по построению графика функции $\sup \Phi^{-1}$, отмеченных выше сразу после \eqref{supf}. В определённой степени аналитическое доказательство нетрудно извлечь из   \cite[гл.~I, \S~4, {\bf 3}, Замечание]{Bour}.

\subsection{Основные неравенства для разности функций}  Цель  --- неравенства между различными усреднениями этих функций. Напомним, что положительная {\it мера\/ $\nu$ на\/ $X$ сосредоточена} на, вообще говоря,  неоднозначно определенном подмножестве $X_{\nu}\subset X$, если $\nu(X\setminus X_{\nu})=0$ \cite[Введение, \S~1]{L}. Нас будет интересовать ситуация, когда  $f$ в \eqref{inJv}  --- разность функций.
\begin{theorem}\label{JId} Пусть 
\begin{enumerate}[{\rm (1)}]
	\item\label{t3_1}    $\mu$ --- конечная мера на $X$ и $0\neq \mu\leq \nu$ для некоторой положительной  счетно-аддит\-и\-в\-н\-ой меры $\nu$ на $X$; 
\item\label{t3_2}  выполнено одно из двух равносильных условий, {\rm I} или {II},  предложения\/ {\rm\ref{co:sup}}; 
\item\label{t3_3}    $u,v\in L^1(X,\mu)$,  $u(x)-v(x)\in I$
 для почти всех $x\in X$ по $\mu$, а суперпозиция $\Phi\circ (u-v)$  $\nu$-интегрируема на $X$;
\item\label{t3_4}   функция $\Phi \circ (u-v)$ положительна  на подмножестве в $X$, где сосредоточена разность  мер $\nu-\mu\geq 0$, и 
\begin{equation*}
	\frac{1}{\mu(X)}\,\int_{X} \Phi\circ (u-v) \, d\nu\in \im \Phi.
\end{equation*}
  \end{enumerate}
 Тогда 
\begin{equation}\label{inJvd+}
	\frac{1}{\mu(X)}\int_{X} u\,d \mu\leq \frac{1}{\mu(X)}\int_{X} v\,d \mu +
	\sup \Phi^{-1}\left(\,\frac{1}{\mu(X)}\,\int_{X} \Phi\circ (u-v) \, d\nu\right) 
\end{equation}
с тем же примечанием по\/ $\sup$, что и в заключении предложения\/ {\rm \ref{pr:b}}.
\end{theorem}
\begin{proof} Если в последнем интеграле в \eqref{inJvd+} рассмотреть меру  $\mu$,  то ввиду \eqref{t3_2}--\eqref{t3_3}  неравенство \eqref{inJvd+} 
--- вариант иной записи предложения \ref{pr:b} для вероятностной меры  $\frac{1}{\mu(X)} \,\mu$ при $f=u-v$. После этого 
ввиду  \eqref{t3_1} и \eqref{t3_4} в последнем интеграле можно заменить $\mu$ на $\nu$, разве что увеличивая правую часть при $\mu\neq \nu$. Если $\mu=\nu$, то мера $\mu-\nu =0$ и сосредоточена на пустом множестве, а любая функция положительна на пустом множестве. 
\end{proof}
\begin{corollary}\label{cor:i} Пусть $X_0\subset X\subset \RR^n$ --- измеримые по мере Лебега $\lambda$ множества в $\RR^n$, 
\begin{enumerate}[{\rm 1)}]
	\item $\lambda(X_0)<+\infty$; 
	\item функция $\Phi$ такая же, как в п.~\eqref{t3_2} теоремы  {\rm \ref{JId}};
	\item   $u,v\in L^1(X,\lambda)$,  $u(x)-v(x)\in I$ для почти всех $x\in X$, а суперпозиция $\Phi\circ (u-v)$  интегрируема на $X$;
\item  функция $\Phi \circ (u-v)$ положительна  на подмножестве в $X\setminus X_0$ и 
\begin{equation*}
	\frac{1}{\lambda (X_0)}\,\int_{X} \Phi\circ (u-v) \, d\lambda \in \im \Phi.
\end{equation*}
\end{enumerate}
Тогда 
\begin{equation}\label{inJvd+l}
	\frac{1}{\lambda(X_0)}\int_{X_0} u\,d \lambda \leq \frac{1}{\lambda(X_0)}\int_{X_0} v\,d \lambda +
	\sup \Phi^{-1}\left(\,\frac{1}{\lambda(X_0)}\,\int_{X} \Phi\circ (u-v) \, d\lambda\right) 
\end{equation}
с тем же примечанием по\/ $\sup$, что и в заключении предложения\/ {\rm \ref{pr:b}}.
\end{corollary}
\begin{proof} Достаточно в теореме \ref{JId} выбрать $\nu=\lambda$ и $\mu=\lambda\bigm|_{X_0}$.
\end{proof} 
\subsection{Разделение переменных} В конкретных ситуациях  два интегральных средних от функций $u$ и $v$ в неравенстве 
\eqref{inJvd+}  из теоремы \ref{JId}  и в неравенстве \eqref{inJvd+l} из следствия  \ref{cor:i},  когда эти средние  носят <<скользящий>>  характер, т.\,е.   может меняться как мера $\mu$, так и  множество $X_0$, ---  <<удобная часть>> оценок \eqref{inJvd+} и   \eqref{inJvd+l}.  Для контроля же последнего  слагаемого в правых частях  \eqref{inJvd+} или   \eqref{inJvd+l}  полезно разделение двух  сомножителей в аргументе функции $\sup \Phi^{-1}$: множителя $\frac{1}{\mu(X)}$ или $\frac{1}{\lambda(X_0)}$  и интеграла. 
 Этого   можно достичь, например, наложением верхних ограничений на функцию $\sup \Phi^{-1}$  мультипликативно-аддити\-в\-н\-о\-го характера: {\it для двух функций $\psi_1\colon (0,+\infty)\to \RR$  и $\psi_2\colon \RR \to \RR$  для  всех произведений\/ $y_1y_2\in \im \Phi$ c $y_1>0$ выполнено  $(\psi_1,\psi_2)$-верхнее условие\/} (относительно операции умножения)
  \begin{enumerate}
	 \item[{\rm (p)}] {\it  степенного типа
		\begin{equation}\label{ap}
	\sup \Phi^{-1} (y_1y_2)\leq \psi_1 (y_1)\cdot \psi_2 (y_2) 
	\end{equation}
	{\rm или}
	\item[{\rm (l)}] 
 логарифмического типа\/} 
		\begin{equation}\label{al}
	\sup \Phi^{-1} (y_1y_2)\leq \psi_1 (y_1)+\psi_2 (y_2) .
		\end{equation}
\end{enumerate}
Из этих определений сразу следуют 
\begin{propos}\label{JId2} В условиях теоремы\/ {\rm\ref{JId}} при выполнении 
$(\psi_1,\psi_2)$-верхнего условия степенного типа {\rm(p)}--\eqref{ap}  
 выполнено неравенство 
	\begin{subequations}\label{epsi}
	\begin{align}
	\frac{1}{\mu(X)}\int_{X} u\,d \mu &\leq \frac{1}{\mu(X)}\int_{X} v\,d \mu +
		\psi_1\left(\,\frac{1}{\mu(X)}\right)\cdot \psi_2\left(\,\int_{X} \Phi\circ (u-v) \, d\nu\right) ,
	\tag{\ref{epsi}p}\label{epsip}\\ 
	\intertext{а $(\psi_1,\psi_2)$-верхнего условия  логарифмического типа {\rm(l)}--\eqref{al}  ---}
	\frac{1}{\mu(X)}\int_{X} u\,d \mu&\leq \frac{1}{\mu(X)}\int_{X} v\,d \mu
		+		 \psi_1 \left(\,\frac{1}{\mu(X)}\right)+  \psi_2 \left(\,\int_{X} \Phi\circ (u-v) \, d\nu\right).
	\tag{\ref{epsi}l}\label{df:N}
	\end{align}
\end{subequations}
\end{propos}

\begin{propos}\label{pr:lam} В условиях следствия\/  {\rm \ref{cor:i}}  при выполнении  $(\psi_1,\psi_2)$-верхнего условия степенного типа {\rm(p)}--\eqref{ap}  выполнено неравенство 
\begin{subequations}\label{epsil}
	\begin{align}
	\frac{1}{\lambda(X_0)}\int_{X_0} u\,d \lambda &\leq \frac{1}{\lambda(X_0)}\int_{X_0} v\,d \lambda +
	\psi_1\left(\,\frac{1}{\lambda(X_0)}\right) \cdot \psi_2 \left(\int_{X} \Phi\circ (u-v) \, d\lambda\right) 
\tag{\ref{epsil}p}\label{epsilp}\\ 
	\intertext{а $(\psi_1,\psi_2)$-верхнего условия  логарифмического типа {\rm(l)}--\eqref{al}  ---}
	\frac{1}{\lambda(X_0)}\int_{X_0} u\,d \lambda &\leq \frac{1}{\lambda(X_0)}\int_{X_0} v\,d \lambda +
	\psi_1\left(\,\frac{1}{\lambda(X_0)}\right) +\psi_2\left(\int_{X} \Phi\circ (u-v) \, d\lambda\right) 
\tag{\ref{epsil}l}\label{df:Nl}
	\end{align}
\end{subequations}
\end{propos}

\begin{example}\label{expp} Пусть  число $p\geq 1$. Применим сформулированные выше результаты к степенной функции  $\Phi \colon t\mapsto (t^+)^p$, $x\in \RR$. 
Тогда для функции $\sup \Phi^{-1}\colon y\to y^{1/p}$ и для функций  $\psi_1=\psi_2 \colon y\mapsto   y^{1/p}$, $y\in \RR^+=\im \Phi$,  выполнено условие (p)  и даже с равенством в \eqref{ap}, т.\,е.  функция $\sup \Phi^{-1}$ удовлетворяет $(\psi_1,\psi_2)$-верхнему условию  степенного типа. В частности, в условиях теоремы  \ref{JId} неравенство  \eqref{epsip} предложения \ref{JId2} записывается в виде
\begin{subequations}\label{inJvdpe}
\begin{align}
	\frac{1}{\mu(X)}\int_{X} u\,d \mu&\leq \frac{1}{\mu(X)}\int_{X} v\,d \mu +
	\frac{1}{\bigl(\mu(X)\bigr)^{1/p}} \left(\,\int_{X} \bigl((u-v)^+\bigr)^p\, d\nu\right)^{1/p},
\tag{\ref{inJvdpe}p}\label{inJvdpep}\\
\intertext{а в условиях  следствия \ref{cor:i}  неравенство \eqref{epsilp}  предложения \ref{epsil}  --- в виде}
		\frac{1}{\lambda(X_0)}\int_{X_0} u\,d \lambda &\leq \frac{1}{\lambda(X_0)}\int_{X_0} v\,d \lambda +
	\left(\,\frac{1}{\lambda(X_0)}\right)^{1/p}\cdot \left(\int_{X} \bigl((u-v)^+\bigr)^p \, d\lambda\right)^{1/p}. 
\tag{\ref{inJvdpe}r}\label{inJvdpee}
\end{align}
		\end{subequations}
 Меняя местами  функции $u$ и $v$  из \eqref{inJvdpe} и соответствующим образом дополняя условия теоремы  \ref{JId} и следствия \ref{cor:i}, из неравенств \eqref{inJvdpe}  получаем  частные случаи неравенства Гёльдера\footnote{Вывод общего неравенства Гёльдера из  неравенства Йенсена см., например, в \cite{PT}, \cite{LL}.}
\begin{subequations}
\begin{align*}
	\frac{1}{\mu(X)}\,\left|\int_{X} (u-v)\,d \mu \right|&\leq 
	\frac{1}{\bigl(\mu(X)\bigr)^{1/p}} \,\left(\int_{X} |u-v|^p\, d\nu\right)^{1/p}, 
	\\
		\frac{1}{\lambda(X_0)} \left|\int_{X_0} (u-v)\,d \lambda\right| &\leq 
		\frac{1}{\bigl(\lambda(X_0)\bigr)^{1/p}} \left(\int_{X} |u-v|^p \, d\lambda\right)^{1/p}. 
\end{align*}
		\end{subequations}
\end{example}

\begin{example}\label{expe} Пусть   $  p>0$. Применим сформулированные выше результаты к показательной строго возрастающей  функции  $\Phi \colon x\to e^{px}$, $x\in \RR$. Тогда $\sup$-обратная функция $\sup \Phi^{-1}\colon y\to \frac{1}{p}\,{\,\ln\,} y$ и для функций  $\psi_1=\psi_2\colon y\to \frac{1}{p}\,{\,\ln\,} y$, $y\in \RR^+=\im \Phi$  выполнено условие (l)  и даже с равенством в \eqref{al}, т.\,е.  функция $\sup \Phi^{-1}$ удовлетворяет $(\psi_1,\psi_2)$-верхнему условию  логарифмического  типа. В частности, в условиях теоремы  \ref{JId} неравенство  \eqref{epsip} предложения \ref{JId2} записывается в виде 
\begin{subequations}\label{inJvdpe2el}
\begin{align}
				\frac{1}{\mu(X)}\int_{X} u\,d \mu&\leq \frac{1}{\mu(X)}\int_{X} v\,d \mu
			+		\frac{1}{p}\, {\,\ln\,} \frac{1}{\mu(X)} +\frac{1}{p} \,{\,\ln\,}\left(\,\int_{X} e^{u-v} \, d\nu\right),
\tag{\ref{inJvdpe2el}l}\label{inJvdpe2epl}
\\
\intertext{а в условиях  следствия \ref{cor:i}  неравенство \eqref{epsilp}  предложения \ref{epsil}  --- в виде}
		\frac{1}{\lambda(X_0)}\int_{X_0} u\,d \lambda&\leq \frac{1}{\lambda(X_0)}\int_{X_0} v\,d \lambda
			+		\frac{1}{p}\, {\,\ln\,} \frac{1}{\lambda(X_0)} +\frac{1}{p} \,{\,\ln\,}\left(\,\int_{X} e^{u-v} \, d\lambda\right)\,.
\tag{\ref{inJvdpe2el}r}\label{inJvdpe2eel}
\end{align}
		\end{subequations}
\end{example}

\section{Доказательства теорем раздела \ref{s10}}\label{s4}
\subsection{К теореме \ref{th:pfvv}}\label{4_1} Условие строгого возрастания  $\Phi$ в теореме \ref{th:pfvv} снимает более общая
\begin{theorem}\label{th:pfv-}
Для  выпуклой функции $\Phi\colon I\to \RR$, удовлетворяющей условию \eqref{t3_2} теоремы\/ {\rm  \ref{JId},} 
при условии $\Phi \circ \bigl({\,\ln\,} |f|- v\bigr)\in L^1(\mathcal O,\lambda)$, т.\,е. при \eqref{fv}, с весом
 $v$, локально интегрируемым в $\mathcal O$ по мере Лебега $\lambda$, имеет место оценка  
\begin{equation}\label{esusr++}
	{\,\ln\,}\bigl|f(z)\bigr|\overset{\eqref{df:B}}{\leq} \inf_{0<r<\dist (z,\CC^n\setminus \mathcal O)}
	\left(B_v(z,r)			+	\sup \Phi^{-1} \Bigl( \frac{\pi^n}{n!}r^{2n} N_{\Phi}(f;v)\Bigr)\right).
	\end{equation}
\end{theorem}

\begin{proof} К левой части \eqref{fv}  можно применить неравенство \eqref{inJvd+l} из следствия \ref{cor:i}
с $X:=\mathcal O$, $X_0:=B(z,r)$, $u:=p{\,\ln\,} |f|$, затем, для левой части полученного неравенства использовать неравенство о среднем по шарам
\begin{equation*}
	p{\,\ln\,}\bigl|f(z)\bigr|\leq  	\frac{1}{\lambda\bigl(B(z,r)\bigr)}\int_{B(z,r)} p{\,\ln\,} |f|\,d \lambda
	\overset{\eqref{df:B}}{=} p B_{ \ln |f|} (z,r)
\end{equation*}
и поделить обе части на $p$.
\end{proof}
\begin{remark}\label{r:rr} 
На примере функций $f\in \Hol (\mathcal O)$ на  открытом множестве $\mathcal O\subset \CC^n$, удовлетворяющих интегральному ограничению  \eqref{df:cle} распространенный способ  получения поточечных оценок для функции $|f|$ состоит 
в получении из  субгармоничности  $|f|^p$ оценки
\begin{equation}\label{essupf}
	\bigl|f(z)\bigr|^p\leq \frac{1}{\lambda\bigl(B(z,r)\bigr)}\int_{B(z,r)} |f|^p\,d \lambda 
	\leq \frac{n!}{\pi^nr^{2n}} \, \exp \biggl(\,\sup_{B(z,r)}v\biggr)\int_{\mathcal O} |f|^p e^{-v}\,d \lambda
	\end{equation}
для любых  $z\in \Omega$ и  $0<r<\dist(z,\CC^n\setminus \mathcal O)$, или в эквивалентной форме, после логарифмирования неравенства \eqref{essupf},
\begin{equation}\label{essupfl}
		{\,\ln\,} \bigl|f(z)\bigr| \leq  \frac{1}{p} \inf_{0<r<\dist(z,\CC^n\setminus \mathcal O)}\left(\sup_{B(z,r)}v +2n{\,\ln\,} \frac1{r}\right) +\const( f,p,v).
\end{equation}
 
Эта оценка практически не учитывает возможных дополнительных свойств функции $v$, например,
гармоничность  на определенных открытых подмножествах из $\mathcal O$ весовой функции $v$.
Очевидно, для {\it любой\/}  весовой  функции  $v$ её усреднение по шару $B(z,r)$ всегда удовлетворяет неравенству
\begin{equation*}
	B_v(z,r)\overset{\eqref{df:B}}{=}\frac{n!}{\pi^nr^{2n}}\int\limits_{B(z,r)} v\,d \lambda {\leq} \sup_{B(z,r)} v,
\end{equation*}
и очень часто строгому. Другими словами, оценка \eqref{esusr}  теремы  \ref{est:hf} не слабее  оценок \eqref{essupf}--\eqref{essupfl}, а в определённых ситуациях и строго сильнее. Это показывают уже следующий
\end{remark}
\begin{example}\label{ex:c} 
Для простоты пусть $n=1$,  
	$\mathcal O:=\Bigl\{z\in \CC \colon \Im z>0 \Bigl\}=:\CC_{+}$
--- верхняя полуплоскость, $v(z)=\Im z$, $z\in \CC_{+}$. 
Тогда для любой функции $f\in \Hol(\CC_{+})$ из ограничения 
\begin{equation}\label{inIM}
	\int\limits_{\CC_{+}} 	\bigl|f(z)\bigr|e^{-\Im z}\, d\lambda (z) 	=
	\iint\limits_{\RR\times \RR^+} \bigl|f(x+iy)\bigr| e^{-y}\, dx \, dy<+\infty
\end{equation}
 по неравенству \eqref{essupfl} имеем  
	\begin{multline}\label{essup1}
		{\,\ln\,} \bigl|f(z)\bigr| \leq  \inf_{0<r<\Im z}\left(\sup_{w\in {B(z,r)}} \Im w +2{\,\ln\,} \frac1{r}\right) +\const( f)
=\Im z+\inf_{0<r<\Im z}\left(r+2{\,\ln\,} \frac1{r}\right)\\+\const( f)=\Im z +\const (f)+
\begin{cases}
	2+2{\,\ln\,} \frac12 \; &\text{при $\Im z \geq 2$},\\
	\Im z+2 \ln\,\textit{}\frac{1}{\Im z} \; &\text{при $0<\Im z < 2$.}
\end{cases}
\end{multline}
В то же время  ограничение \eqref{inIM} по   оценке \eqref{esusr} теоремы  \ref{est:hf} в силу гармоничности весовой функции даёт
\begin{equation}\label{esusrim}
	{\,\ln\,}\bigl|f(z)\bigr|\leq \inf_{0<r<\Im z}\left(\Im z	+	{2}\, {\,\ln\,} \frac{1}{r}\right) +\const( f)
	=\Im z	+	{2}\, {\,\ln\,} \frac{1}{\Im z} +\const( f)
	\end{equation}
при всех $z\in \CC_+$.Правая часть в  \eqref{esusrim} всегда строго меньше правой части \eqref{essup1} и даже растёт в определённом смысле медленнее на  слагаемое $-2\, {\,\ln\,} {\Im z}$ при $\Im z >2$. 
\end{example}

\subsection{Доказательство теоремы \ref{cor:iH}}
Будет использована
\begin{theoremH}[{\cite[теорема 4.2.6]{H}}]\label{th:H} В условиях теоремы\/ {\rm \ref{cor:iH}} существует  решение $f$, уравнение 
$\bar\partial f=g$  имеет решение $f\in L^2_{\loc}(\mathcal O)$, удовлетворяющее интегральной оценке
\begin{equation}\label{est:ufe}
\int\limits_{\mathcal O}\bigl|f(z) \bigr|^2e^{-v(z)}\bigl(1+|z|^2\bigr)^{-a} 
\, d \lambda (z)
\overset{\eqref{co:r}}{\leq} 
\frac{1}{a}\, J(g,v).
\end{equation}
\end{theoremH}

Перепишем левую часть \eqref{est:ufe} в виде
\begin{equation}\label{lp:in}
 \int_{\mathcal O} \exp \Bigl(2{\,\ln\,} \bigl|f(z)\bigr|-w(z)\Bigr) \, d \lambda (z),
\end{equation}
где 
\begin{equation}\label{df:w}
	w(z)=v(z)+a {\,\ln\,} \bigl(1+|z|^2\bigr). 
\end{equation}
Для оценки снизу величины \eqref{lp:in} воспользуемся  
неравенством \eqref{inJvdpe2eel}  с $X:=\mathcal O$, $X_0:=B(z,r)\Subset \mathcal  O$ в условиях \eqref{co:rd},  $p=1$, $u:=2{\,\ln\,} |f|$ и с $w$ вместо $v$:
\begin{multline}\label{e:klon}
2B_{\ln |f|}( z,r) \leq B_w(z,r) +		{\,\ln\,} \frac{1}{\lambda\bigl(B(z,r)\bigr)} + \,{\,\ln\,}\left(\,\int\limits_{\mathcal O} e^{2{\,\ln\,} |f|-w}
			\, d\lambda\right)
			\\ 		\overset{\eqref{df:w}}{=} B_v\bigl(z,r\bigr)  +a  \int\limits_{B(z,r)} 
			{\,\ln\,} \bigl(1+|z'|^2\bigr) 
			\,d \lambda (z') 
	+{\,\ln\,} \frac{n!}{\pi^nr^{2n}}+{\,\ln\,} \frac{1}{a} J(g,v,a), 
\\
\overset{\eqref{co:rd}}{\leq} B_v\bigl(z,r\bigr)  +2a {\,\ln\,} \bigl(1+|z|\bigr) 
				+{\,\ln\,} \frac{n!}{\pi^nr^{2n}}+{\,\ln\,} J(g,v) +\const(a),
\end{multline}
что и даёт  требуемую оценку \eqref{efvaz}.

\fullauthor{Хабибуллин Булат Нурмиевич}
\address{450074, г. Уфа, ул. З. Валиди, 32, БашГУ, ФМиИТ, заведующий кафедрой высшей алгебры и геометрии, профессор}
\email{khabib-bulat@mail.ru}

\fullauthor{Khabibullin Bulat Nurmievich}
\address{Prof., Head of the Chair of Higher Algebra and Geometry, Dept. of Math. \& IT, Bash. State Univ., Z. Validi Str., 32,  Ufa, Bashkortostan, 450076, Russian Federation}
\email{khabib-bulat@mail.ru}

\fullauthor{Баладай Рустам Алексеевич}
\address{450076, г. Уфа, ул. З. Валиди, 32, БашГУ, ФМиИТ, кафедра высшей алгебры и геометрии, аспирант}
\email{rbaladai@gmail.com}

\fullauthor{Baladai Rustam Alekseevich}
\address{Postgraduate of the Chair of Higher Algebra and Geometry, Dept. of Math. \& IT, Bash. State Univ., Z. Validi Str., 32, Ufa, Bashkortostan, 450076, Russian Federation}
\email{rbaladai@gmail.com}





\label{lastpage}  

\end{document}